\documentclass[a4paper,10pt]{amsart}
\usepackage{amsmath,amssymb,amsthm}
\usepackage{epsfig}
\usepackage{graphicx}
\newcommand{\CM}{\overline{\operatorname{\mathcal{M}}}}

\newcommand{\LL}{\operatorname{\mathcal{L}}}

\newcommand{\PP}{\operatorname{\mathfrak{P}}}
\newcommand{\WW}{\operatorname{\mathfrak{W}}}

\newcommand{\DD}{\operatorname{\mathfrak{D}}}

\newcommand{\Si}{\dot{S}}

\newcommand{\ev}{\operatorname{ev}}

\newcommand{\PD}{\operatorname{PD}}

\newcommand{\CZ}{\operatorname{CZ}}

\newcommand{\del}{\partial}

\newcommand{\Ju}{\underline{J}}

\newcommand{\IC}{\operatorname{\mathbb{C}}}

\newcommand{\IR}{\operatorname{\mathbb{R}}}
\newcommand{\IN}{\operatorname{\mathbb{N}}}
\newcommand{\ID}{\operatorname{\mathbf{D}}}
\newcommand{\IH}{\operatorname{\mathbf{H}}}

\newcommand{\Ih}{\operatorname{\mathbf{h}}}

\newcommand{\IF}{\operatorname{\mathbf{F}}}

\newtheorem{theorem}{Theorem}[section]

\newtheorem{corollary}[theorem]{Corollary}

\title{String, dilaton and divisor equation\\ in Symplectic Field Theory}
\author{Oliver Fabert and Paolo Rossi}
\begin{document}

\maketitle

\begin{abstract}
Infinite dimensional Hamiltonian systems appear naturally in the rich algebraic structure of Symplectic Field Theory. 
Carefully defining a generalization of gravitational descendants and adding them to the picture, one can produce an 
infinite number of symmetries of such systems . As in Gromov-Witten theory, 
the study of the topological meaning of gravitational descendants yields new differential equations for the 
SFT Hamiltonian, where the key point is to understand the dependence of the algebraic constructions 
on choices of auxiliary data like contact form, cylindrical almost complex structure, abstract perturbations, 
differential forms and coherent collections of sections used to define gravitational descendants.
\end{abstract}

\tableofcontents

\markboth{O. Fabert and P. Rossi}{String equation in SFT} 

\section{Introduction}
Symplectic field theory (SFT), introduced by H. Hofer, A. Givental and Y. Eliashberg in 2000 ([EGH]), is a very large 
project and can be viewed as a topological quantum field theory approach to Gromov-Witten theory. Besides providing a 
unified view on established pseudoholomorphic curve theories like symplectic Floer homology, contact homology and 
Gromov-Witten theory, it leads to numerous new applications and opens new routes yet to be explored. \\
 
While symplectic field theory leads to algebraic invariants with very rich algebraic structures, it was pointed out by Eliashberg in his ICM 2006 plenary talk ([E]) 
that the integrable systems of rational Gromov-Witten theory very naturally appear in rational symplectic field 
theory by using the link between the rational symplectic field theory of prequantization spaces in the 
Morse-Bott version and the rational Gromov-Witten potential of the underlying symplectic manifold, see the recent papers [R1], 
[R2] by the second author. Indeed, after introducing gravitational descendants as in Gromov-Witten theory, it is precisely the rich algebraic 
formalism of SFT with its Weyl and Poisson structures that provides a natural link between symplectic field theory 
and (quantum) integrable systems. \\ 
 
Carefully defining a generalization of gravitational descendants and adding them to the picture, the first author has 
shown in [F] that one can assign to every contact manifold an infinite sequence of commuting Hamiltonian systems on 
SFT homology and the question of their integrability arises. For this it is important to fully understand the algebraic structure of 
gravitational descendants in SFT. While it is well-known that in Gromov-Witten theory the topological meaning of 
gravitational descendants leads to new differential equations for the Gromov-Witten potential, it is interesting to 
ask how these rich algebraic structures carry over from Gromov-Witten theory to symplectic field theory. \\

As a first step, we will show in this paper how the well-known string, dilaton and divisor equations generalize from 
Gromov-Witten theory to symplectic field theory, where the key point is the covariance of the algebraic constructions 
under choices of auxiliary data like contact form, cylindrical almost complex structure, abstract perturbations and 
coherent collections of sections used to define gravitational descendants. It will turn that we obtained the same equations as 
in Gromov-Witten theory (up to contributions of constant curves), but these however only hold after passing to SFT homology. \\

Most of this paper was written when both authors were members of the Mathematical Sciences Research Institute (MSRI) in Berkeley 
and it was finished when the first author was a postdoc at the Max Planck Institute (MPI) for Mathematics in the Sciences in Germany 
and the second author was a postdoc at the Institut de Mathematiques de Jussieu, Paris VI. They want to thank the institutes for their hospitality and their great working environment. Further they want to thank 
Y. Eliashberg, A. Givental, J. Latschev and D. Zvonkine for useful discussions.
 
\section{SFT and commuting quantum Hamiltonian systems}
Symplectic field theory (SFT) is a very large project, initiated by Eliashberg,
Givental and Hofer in their paper [EGH], designed to describe in a unified way 
the theory of pseudoholomorphic curves in symplectic and contact topology. 
Besides providing a unified view on well-known theories like symplectic Floer 
homology and Gromov-Witten theory, it shows how to assign algebraic invariants 
to closed contact manifolds $(V,\xi=\{\lambda=0\})$: \\
   
Recall that a contact one-form $\lambda$ defines a vector field $R$ on $V$ by 
$R\in\ker d\lambda$ and $\lambda(R)=1$, which 
is called the Reeb vector field. We assume that 
the contact form is Morse in the sense that all closed orbits of the 
Reeb vector field are nondegenerate in the sense of [BEHWZ]; in particular, the set 
of closed Reeb orbits is discrete. The invariants are defined by counting 
$\Ju$-holomorphic curves in $\IR\times V$ which are asymptotically cylindrical over 
chosen collections of Reeb orbits $\Gamma^{\pm}=\{\gamma^{\pm}_1,...,
\gamma^{\pm}_{n^{\pm}}\}$ as the $\IR$-factor tends to $\pm\infty$, see [BEHWZ]. 
The almost complex structure $\Ju$ on the cylindrical 
manifold $\IR\times V$ is required to be cylindrical in the sense that it is  
$\IR$-independent, links the two natural vector fields on $\IR\times V$, namely the 
Reeb vector field $R$ and the $\IR$-direction $\del_s$, by $\Ju\del_s=R$, and turns 
the distribution $\xi$ on $V$ into a complex subbundle of $TV$, 
$\xi=TV\cap \Ju TV$. We denote by $\CM_{g,r,A}(\Gamma^+,\Gamma^-)/\IR$ the corresponding compactified
moduli space of genus $g$ curves with $r$ additional marked points representing the absolute homology class 
$A\in H_2(V)$ using a choice of spanning surfaces ([BEHWZ],[EGH]). 
Possibly after choosing abstract perturbations using polyfolds following [HWZ], we get that 
$\CM_{g,r,A}(\Gamma^+,\Gamma^-)$ is a (weighted branched) manifold with corners of dimension 
equal to the Fredholm index of the Cauchy-Riemann operator for $\Ju$. 
{\it Note that as in [F] we will not discuss transversality for the Cauchy-Riemann operator but just refer to the upcoming 
papers on polyfolds by H. Hofer and his co-workers.} \\  
 
Let us now briefly introduce the algebraic formalism of SFT as described in [EGH]: \\
 
Recall that a multiply-covered Reeb orbit $\gamma^k$ is called bad if 
$\CZ(\gamma^k)\neq\CZ(\gamma)\mod 2$, where $\CZ(\gamma)$ denotes the 
Conley-Zehnder index of $\gamma$. Calling a Reeb orbit $\gamma$ {\it good} if it is not bad we assign to every 
good Reeb orbit $\gamma$ two formal graded variables $p_{\gamma},q_{\gamma}$ with grading 
\begin{equation*} 
|p_{\gamma}|=m-3-\CZ(\gamma),|q_{\gamma}|=m-3+\CZ(\gamma) 
\end{equation*} 
when $\dim V = 2m-1$. Assuming we have chosen a basis $A_0,...,A_N$ of $H_2(V)$, we assign to every $A_i$ a formal 
variables $z_i$ with grading $|z_i|=- 2 c_1(A_i)$. In order to include higher-dimensional moduli spaces we further assume that a string 
of closed (homogeneous) differential forms $\Theta=(\theta_1,...,\theta_N)$ on $V$ is chosen and assign to 
every $\theta_{\alpha}\in\Omega^*(V)$ a formal variables $t_{\alpha}$ 
with grading
\begin{equation*} |t_{\alpha}|=2 -\deg\theta_{\alpha}. \end{equation*}  
Finally, let $\hbar$ be another formal variable of degree $|\hbar|=2(m-3)$. \\

Let $\WW$ be the graded Weyl algebra over $\IC$ of power series in the variables 
$\hbar,p_{\gamma}$ and $t_i$ with coefficients which are polynomials in the 
variables $q_{\gamma}$ and $z_n$, which is equipped with the associative product $\star$ in 
which all variables super-commute according to their grading except for the 
variables $p_{\gamma}$, $q_{\gamma}$ corresponding to the same Reeb orbit $\gamma$, 
\begin{equation*} [p_{\gamma},q_{\gamma}] = 
                  p_{\gamma}\star q_{\gamma} -(-1)^{|p_{\gamma}||q_{\gamma}|} 
                  q_{\gamma}\star p_{\gamma} = \kappa_{\gamma}\hbar.
\end{equation*}
($\kappa_{\gamma}$ denotes the multiplicity of $\gamma$.) Since it is shown in [EGH] that the bracket 
of two elements in $\WW$ gives an element in $\hbar\WW$, it follows that we get a bracket on the module 
$\hbar^{-1}\WW$. Following [EGH] we further introduce 
the Poisson algebra $\PP$ of formal power series in the variables $p_{\gamma}$ and $t_i$ with   
coefficients which are polynomials in the variables $q_{\gamma}$ with Poisson bracket given by 
\begin{equation*} 
 \{f,g\} = \sum_{\gamma}\kappa_{\gamma}\Bigl(\frac{\del f}{\del p_{\gamma}}\frac{\del g}{\del q_{\gamma}} -
                          (-1)^{|f||g|}\frac{\del g}{\del p_{\gamma}}\frac{\del f}{\del q_{\gamma}}\Bigr).       
\end{equation*}

As in Gromov-Witten theory we want to organize all moduli spaces $\CM_{g,r,A}(\Gamma^+,\Gamma^-)$
into a generating function $\IH\in\hbar^{-1}\WW$, called {\it Hamiltonian}. In order to include also higher-dimensional 
moduli spaces, in [EGH] the authors follow the approach in Gromov-Witten theory to integrate the chosen differential forms 
$\theta_{\alpha}$ over the moduli spaces after pulling them back under the evaluation map from target manifold $V$. 
The Hamiltonian $\IH$ is then defined by
\begin{equation*}
 \IH = \sum_{\Gamma^+,\Gamma^-} \int_{\CM_{g,r,A}(\Gamma^+,\Gamma^-)/\IR}
 \ev_1^*\theta_{\alpha_1}\wedge...\wedge\ev_r^*\theta_{\alpha_r}\; \hbar^{g-1}t^Ip^{\Gamma^+}q^{\Gamma^-}z^d
\end{equation*}
with $t^{\alpha}=t_{\alpha_1}...t_{\alpha_r}$, $p^{\Gamma^+}=p_{\gamma^+_1}...p_{\gamma^+_{n^+}}$, 
$q^{\Gamma^-}=q_{\gamma^-_1}...q_{\gamma^-_{n^-}}$ and $z^d = z_0^{d_0} \cdot ... \cdot z_N^{d_N}$. 
Expanding 
\begin{equation*} \IH=\hbar^{-1}\sum_g \IH_g \hbar^g \end{equation*} 
we further get a rational Hamiltonian $\Ih=\IH_0\in\PP$, which counts only curves with genus zero. \\

While the Hamiltonian $\IH$ explicitly depends on the chosen contact form, the cylindrical almost complex structure, 
the differential forms and abstract polyfold perturbations making all moduli spaces regular, it is outlined in [EGH] 
how to construct algebraic invariants, which just depend on the contact structure and the cohomology classes of the 
differential forms. \\
  
In complete analogy to Gromov-Witten theory we can introduce $r$ tautological line 
bundles $\LL_1,...,\LL_r$ over each moduli space $\CM_r=\CM_{g,r,A}(\Gamma^+,\Gamma^-)/\IR$ , where the fibre of $\LL_i$ 
over a punctured curve $(u,\Si)\in\CM_r$ is again given 
by the cotangent line to the underlying, possibly unstable nodal Riemann surface (without ghost components) at the 
$i$.th marked point and which again formally can be defined as the pull-back of the vertical cotangent line 
bundle of $\pi: \CM_{r+1}\to\CM_r$ under the canonical section $\sigma_i: \CM_r\to\CM_{r+1}$ mapping to the $i$.th marked 
point in the fibre. Note again that while the vertical cotangent line bundle is rather a sheaf (the dualizing sheaf) than a true bundle since 
it becomes singular at the nodes in the fibres, the pull-backs under the canonical sections are still true line bundles 
as the marked points are different from the nodes and hence these sections avoid the singular loci. \\

While in Gromov-Witten theory the gravitational descendants were defined by integrating powers of the first Chern class 
of the tautological line bundle over the moduli space, which by Poincare duality corresponds to counting common zeroes of 
sections in this bundle, in symplectic field theory, more generally every holomorphic curves theory where curves with 
punctures and/or boundary are considered, we are faced with the problem that the moduli spaces generically have 
codimension-one boundary, so that the count of zeroes of sections in general depends on the chosen sections in the 
boundary. It follows that the integration of the first Chern class of the tautological line bundle over a single moduli 
space has to be replaced by a construction involving all moduli space at once. Note that this is similar to the choice of 
coherent abstract perturbations for the moduli spaces in symplectic field theory in order to achieve transversality for 
the Cauchy-Riemann operator. \\

Keeping the interpretation of descendants as common zero sets of sections in powers of the 
tautological line bundles, the first author defined in his paper [F] the notion of {\it coherent collections of sections} 
$(s)$ in the tautological line bundles over all moduli spaces, which just formalizes how the sections chosen for the 
lower-dimensional moduli spaces should affect the section chosen for a moduli spaces on its boundary. Based on this he then 
defined {\it descendants of moduli spaces} $\CM^j\subset\CM$, which were obtained inductively as zero sets of these coherent 
collections of sections $(s_j)$ in the tautological line bundles over the descendant moduli spaces $\CM^{j-1}\subset\CM$. \\

So far we have only considered the case with one additional marked point. On the other hand, as already outlined in [F], 
the general case with $r$ additional marked points is just notationally more involved. Indeed, we can 
easily define for every moduli space $\CM_r=\CM_{g,r,A}(\Gamma^+,\Gamma^-)/\IR$ with $r$ additional marked points and every 
$r$-tuple of natural numbers $(j_1,...,j_r)$ descendants $\CM^{(j_1,...,j_r)}_r\subset\CM_r$ by setting
\begin{equation*} \CM^{(j_1,...,j_r)}_r = \CM^{(j_1,0,...,0)}_r\cap ... \cap \CM^{(0,...,0,j_r)}_r, \end{equation*}
where the descendant moduli spaces $\CM^{(0,...,0,j_i,0,...,0)}_r\subset\CM_r$ are defined in the same way as the 
one-point descendant moduli spaces $\CM^{j_i}_1\subset\CM_1$ by looking at the $r$ tautological line bundles $\LL_{i,r}$ 
over the moduli space $\CM_r = \CM_r(\Gamma^+,\Gamma^-)/\IR$ separately. In other words, we inductively choose generic 
sections $s^j_{i,r}$ in the line bundles $\LL_{i,r}^{\otimes j}$ to define $\CM^{(0,...,0,j,0,...,0)}_r=
(s^j_{i,r})^{-1}(0)\subset\CM^{(0,...,0,j-1,0,...,0)}_r\subset\CM_r$.  \\

With this we can define the descendant Hamiltonian of SFT, which we will continue denoting by $\IH$, while the 
Hamiltonian defined in [EGH] will from now on be called {\it primary}. In order to keep track of the descendants we 
will assign to every chosen differential form $\theta_i$ now a sequence of formal variables $t_{i,j}$ with grading 
\begin{equation*} |t_{i,j}|=2(1-j) -\deg\theta_i. \end{equation*} 
Then the descendant Hamiltonian $\IH\in\hbar^{-1}\WW$ of SFT is defined by 
\begin{equation*}
 \IH = \sum_{\Gamma^+,\Gamma^-,I} \int_{\CM^{(j_1,...,j_r)}_{g,r,A}(\Gamma^+,\Gamma^-)/\IR}
 \ev_1^*\theta_{\alpha_1}\wedge...\wedge\ev_r^*\theta_{\alpha_r}\; \hbar^{g-1}t^Ip^{\Gamma^+}q^{\Gamma^-},
\end{equation*}
where $p^{\Gamma^+}=p_{\gamma^+_1} ... p_{\gamma^+_{n^+}}$, $q^{\Gamma^-}=q_{\gamma^-_1} ... q_{\gamma^-_{n^-}}$ and 
$t^{\alpha,j}=t_{\alpha_1,j_1} ... t_{\alpha_r,j_r}$.\\

We want to emphasize that the following statement is not yet a theorem in the strict mathematical sense as the analytical 
foundations of symplectic field theory, in particular, the neccessary transversality theorems for the Cauchy-Riemann 
operator, are not yet fully established. Since it can be expected that the polyfold project by Hofer and his 
collaborators sketched in [HWZ] will provide the required transversality theorems, we follow other papers in the field in 
proving everything up to transversality and state it nevertheless as a theorem. \\
\\
{\bf Theorem:} {\it Differentiating the Hamiltonian $\IH\in\hbar^{-1}\WW$ with respect to the formal variables $t_{\alpha,p}$ 
defines a sequence of quantum Hamiltonian} 
\begin{equation*} \IH_{\alpha,p}=\frac{\del\IH}{\del t^{\alpha,p}} \in H_*(\hbar^{-1}\WW,[\IH,\cdot]) \end{equation*} 
{\it in the full SFT homology algebra with differential $D=[\IH,\cdot]: \hbar^{-1}\WW\to\hbar^{-1}\WW$, which commute with respect to the 
bracket on $H_*(\hbar^{-1}\WW,[\IH,\cdot])$,} 
\begin{equation*} [\IH_{\alpha,p},\IH_{\beta,q}] = 0,\; (\alpha,p),(\beta,q)\in\{1,...,N\}\times\IN. \end{equation*}
\\
Everything is an immediate consequence of the master equation $[\IH,\IH]=0$, which can be proven in the same 
way as in the case without descendants using the results in [F]. While the boundary equation $D\circ D=0$ is well-known 
to follow directly from the identity $[\IH,\IH]=0$, the fact that every $\IH_{\alpha,p}$, $(\alpha,p)\in
\{1,...,N\}\times\IN$ defines an element in the homology $H_*(\hbar^{-1}\WW,[\IH,\cdot])$ follows from the identity
\begin{equation*} [\IH,\IH_{\alpha,p}] = 0,\end{equation*} 
which can be shown by differentiating the master equation with respect to the $t_{\alpha,p}$-variable and using the 
graded Leibniz rule,
\[ \frac{\del}{\del t_{\alpha,p}} [f,g] = 
   [\frac{\del f}{\del t_{\alpha,p}},g] + (-1)^{|t_{\alpha,p}||f|} [f,\frac{\del g}{\del t_{\alpha,p}}]. \]
On the other hand, in order to see that any two $\IH_{\alpha,p}$, 
$\IH_{\beta,q}$ commute {\it after passing to homology} it suffices to see that by differentiating twice (and using that all 
summands in $\IH$ have odd degree) we get the identity
\begin{equation*} 
 [\IH_{\alpha,p},\IH_{\beta,q}]+(-1)^{|t_{\alpha,p}|}[\IH,\frac{\del^2\IH}{\del t_{\alpha,p}\del t_{\beta,q}}] = 0. 
\end{equation*} 

Let $\WW^0$ be the graded Weyl algebra over $\IC$, which is obtained from the big Weyl algebra $\WW$ by setting all 
variables $t_{\alpha,p}$ equal to zero. Apart from the fact that the Hamiltonian $\IH^0 = \IH|_{t=0}\in\hbar^{-1}\WW^0$ now 
counts only curves with no additional marked points, the new SFT Hamiltonians $\IH^1_{\alpha,p}=\IH_{\alpha,p}|_{t=0}
\in\hbar^{-1}\WW^0$, $(\alpha,p)\in\{1,...,N\}\times \IN$ now count holomorphic curves with one marked point. In other words,
specializing at $t=0$ we get back the following theorem proven in [F]. \\
\\
{\bf Theorem:} {\it Counting holomorphic curves with one marked point after integrating differential forms and introducing gravitational descendants defines a sequence of distinguished elements} 
\begin{equation*} \IH^1_{\alpha,p}\in H_*(\hbar^{-1}\WW^0,D^0) \end{equation*} 
{\it in the full SFT homology algebra with differential $D^0=[\IH^0,\cdot]: \hbar^{-1}\WW^0\to\hbar^{-1}\WW^0$, which commute with respect to the bracket on $H_*(\hbar^{-1}\WW^0,D^0)$,} 
\begin{equation*} [\IH^1_{\alpha,p},\IH^1_{\beta,q}] = 0,\; (\alpha,p),(\beta,q)\in\{1,...,N\}\times\IN. \end{equation*}

We now turn to the question of independence of these nice algebraic structures from the choices like contact form, 
cylindrical almost complex structure, abstract polyfold perturbations and, of course, the choice of the coherent 
collection of sections. This is the content of the following theorem, where we however again want to emphasize that 
the following statement is not yet a theorem in the strict mathematical sense as the analytical foundations of symplectic 
field theory, in particular, the neccessary transversality theorems for the Cauchy-Riemann operator, are not yet fully 
established. \\   
\\
{\bf Theorem:} {\it For different choices of contact form $\lambda^{\pm}$, cylindrical almost complex structure 
$\Ju^{\pm}$ , abstract polyfold perturbations and sequences of coherent collections of sections $(s^{\pm}_j)$ the 
resulting systems of commuting operators $\IH^-_{\alpha,p}$ on $H_*(\hbar^{-1}\WW^-,D^-)$ and $\IH^+_{\alpha,p}$ on 
$H_*(\hbar^{-1}\WW^+,D^+)$ are isomorphic, i.e., there exists an isomorphism of the Weyl algebras $H_*(\hbar^{-1}\WW^-,D^-)$ 
and $H_*(\hbar^{-1}\WW^+,D^+)$ which maps $\IH^-_{\alpha,p}\in H_*(\hbar^{-1}\WW^-,D^-)$ to 
$\IH^+_{\alpha,p}\in H_*(\hbar^{-1}\WW^+,D^+)$.} \\

Specializing at $t=0$ we again get back the theorem proven in [F]. \\
\\
{\bf Theorem:} {\it For different choices of contact form $\lambda^{\pm}$, cylindrical almost complex structure 
$\Ju^{\pm}$ , abstract polyfold perturbations and sequences of coherent collections of sections $(s^{\pm}_j)$ the 
resulting systems of commuting operators $\IH^{1,-}_{\alpha,p}$ on $H_*(\hbar^{-1}\WW^{0,-},D^{0,-})$ and $\IH^{1,+}_{\alpha,p}$ on 
$H_*(\hbar^{-1}\WW^{0,+},D^{0,+})$ are isomorphic, i.e., there exists an isomorphism of the Weyl algebras $H_*(\hbar^{-1}\WW^{0,-},D^{0,-})$ 
and $H_*(\hbar^{-1}\WW^{0,+},D^{0,+})$ which maps $\IH^{1,-}_{\alpha,p}\in H_*(\hbar^{-1}\WW^{0,-},D^{0,-})$ to 
$\IH^{1,+}_{\alpha,p}\in H_*(\hbar^{-1}\WW^{0,+},D^{0,+})$.} \\

For the proof observe that in [F] the first author introduced the notion of a collection of sections $(s_j)$ in the 
tautological line bundles over all moduli spaces of holomorphic curves in the cylindrical cobordism interpolating between the
auxiliary structures which are {\it coherently connecting} the two coherent collections of sections $(s^{\pm}_j)$. \\ 

In order to prove the above invariance theorem we now recall the extension of the algebraic formalism of SFT from 
cylindrical manifolds to symplectic cobordisms with cylindrical ends as described in [EGH]. \\

Let $\DD$ be the space of formal power series in the variables $\hbar,p^+_{\gamma}$ with coefficients which are polynomials in the 
variables $q^-_{\gamma}$. Elements in $\WW^{\pm}$ then act as differential operators from the right/left on 
$\DD$ via the replacements 
\begin{equation*} 
  q^+_{\gamma}\mapsto \kappa_{\gamma}\hbar\overleftarrow{\frac{\del}{\del p^+_{\gamma}}},\;\;
  p^-_{\gamma}\mapsto\kappa_{\gamma}\hbar\overrightarrow{\frac{\del}{\del q^-_{\gamma}}}.
\end{equation*}

In the very same way as we followed [EGH] and defined the Hamiltonians $\IH^{\pm}$ counting holomorphic curves in the cylindrical manifolds 
$V^{\pm}$ with contact forms $\lambda^{\pm}$, cylindrical almost complex structures $\Ju^{\pm}$, abstract perturbations and 
coherent collections of sections $(s_j^{\pm})$, we now define a potential $\IF\in\hbar^{-1}\DD$ counting holomorphic 
curves in the symplectic cobordism $W$ between the contact manifolds $V^{\pm}$ with interpolating auxiliary data, in particular, 
using the collection of sections $(s_j)$ coherently connecting $(s_j^{\pm})$. \\

Along the lines of the proof in [EGH], it follows that we have the fundamental identity  
\begin{equation*} e^{\IF}\overleftarrow{\IH^+} - \overrightarrow{\IH^-}e^{\IF} = 0 \end{equation*}
In the same way as in [EGH] this implies that 
\begin{equation*}
  D^{\IF}: \hbar^{-1}\DD\to\hbar^{-1}\DD,\; 
  D^{\IF}g = e^{-\IF}\overrightarrow{\IH^-}(ge^{\IF}) - (-1)^{|g|}(ge^{\IF})\overleftarrow{\IH^+}e^{-\IF} 
\end{equation*}
satisfies $D^{\IF}\circ D^{\IF} = 0$ and hence can be used to define the homology algebra 
$H_*(\hbar^{-1}\DD,D^{\IF})$. Furthermore it is shown that the maps 
\begin{eqnarray*}
  &&F^-: \hbar^{-1}\WW^-\to\hbar^{-1}\DD,\; f\mapsto e^{-\IF}\overrightarrow{f}e^{+\IF}, \\
  &&F^+: \hbar^{-1}\WW^+\to\hbar^{-1}\DD,\; f\mapsto e^{+\IF}\overleftarrow{f}e^{-\IF}
\end{eqnarray*}
commute with the boundary operators,
\begin{equation*} F^{\pm}\circ D^{\pm} = D^{\IF}\circ F^{\pm}, \end{equation*} 
and hence descend to maps between the homology algebras
\begin{equation*} F^{\pm}_*: H_*(\hbar^{-1}\WW^{\pm},D^{\pm})\to H_*(\hbar^{-1}\DD,D^{\IF}), \end{equation*}
where it can be shown as in [EGH] that both maps are isomorphisms if $W=\IR\times V$ and the contact forms $\lambda^{\pm}$ 
induce the same contact structure $\xi = \ker \lambda^{\pm}$. \\ 

On the other hand, differentiating the potential $\IF\in\hbar^{-1}\DD$ and the two Hamiltonians $\IH^{\pm}\in\hbar^{-1}\WW^{\pm}$ with 
respect to the $t_{\alpha,p}$-variables, we get also the identity
\begin{equation*} 
   e^{\IF}\overleftarrow{\IH^+_{\alpha,p}} - \overrightarrow{\IH^-_{\alpha,p}}e^{\IF} = 
   (-1)^{|t_{\alpha,p}|} \overrightarrow{\IH^-}(e^{\IF}\IF_{\alpha,p}) - (e^{\IF}\IF_{\alpha,p})\overleftarrow{\IH^+}, 
\end{equation*} 
about $\IF$, $\IF_{\alpha,p}=\frac{\del\IF}{\del t_{\alpha,p}}$ and $\IH^{\pm}$, $\IH^{\pm}_{\alpha,p}$, where we 
used that all summands in $\IH^{\pm}$ ($\IF$) have odd (even) degree and 
\begin{equation*} 
  \frac{\del}{\del t_{\alpha,p}}e^{\IF}= e^{\IF} \IF_{\alpha,p}. 
\end{equation*} 
On the other hand, it is easy to see that the above identity implies that
\begin{equation*} 
    F^+(\IH^+_{\alpha,p}) - F^-(\IH^-_{\alpha,p}) = 
    e^{+\IF}\overleftarrow{\IH^+_{\alpha,p}}e^{-\IF} - e^{-\IF}\overrightarrow{\IH^-_{\alpha,p}}e^{+\IF} 
\end{equation*}
is equal to  
\begin{equation*}
  (-1)^{|t_{\alpha,p}|} e^{-\IF}\overrightarrow{\IH^-}(e^{+\IF}\IF_{\alpha,p}) - (e^{+\IF}\IF_{\alpha,p})\overleftarrow{\IH^+}e^{-\IF} 
  = (-1)^{|t_{\alpha,p}|} D^{\IF}(\IF_{\alpha,p}),
\end{equation*}
so that, after passing to homology, we have
\begin{equation*}  F^+_*(\IH^+_{\alpha,p}) = F^-_*(\IH^-_{\alpha,p}) \in H_*(\hbar^{-1}\DD,D^{\IF}) \end{equation*}
as desired. \\

\section{Divisor, dilaton and string equations in SFT}
The goal of this paper is to understand how the well-known divisor, dilaton and string equations from Gromov-Witten 
theory generalize to symplectic field theory. Here the main problem is to deal with the fact that the SFT Hamiltonian 
itself is not an invariant for the contact manifold. More precisely it depends not only on choices like contact form, 
cylindrical almost complex structure and coherent abstract perturbations but also on the chosen differential forms 
$\theta_i$ and coherent collections of sections $(s_j)$ used to define gravitational descendants. The main application 
of these equations we have in mind is the computation of the sequence of commuting quantum Hamiltonians 
$\IH_{\alpha,p}=\frac{\del\IH}{\del t^{\alpha,p}}$ on SFT homology $H_*(\hbar^{-1}\WW,D)$ introduced in the last section.

\vspace{0.5cm}

\subsection{Special non-generic coherent collections of sections}
In order to prove the desired equations we will start with special non-generic choices of coherent collections of 
sections in the tautological bundles $\LL_{i,r}$ over all moduli spaces $\CM_r=\CM_{g,r,A}(\Gamma^+,\Gamma^-)/\IR$. \\

The first assumption we will make is about the choice of sections in the tautological line bundles $\LL_{1,1}$ over 
the simplest moduli spaces $\CM_{0,1}(\gamma,\gamma)/\IR\cong S^1$ of orbit cylinders with one marked point. Observing that 
$\LL_{1,1}$ has a natural trivialization by canonically identifying $\CM_{0,1}(\gamma,\gamma)/\IR$ with the target Reeb 
orbit $\gamma$ and the bundle itself with the cotangent bundle to $\IR\times\gamma$, we want to assume that the section in $\LL_{1,1}$ is 
constant in this trivialization. \\

This choice has a nice consequence. For this consider the generic fibre 
$F_{(u,\Si)}=\pi_r^{-1}((u,\Si)) \in \CM_{g,r,A}(\Gamma^+,\Gamma^-)/\IR$ of the forgetful fibration 
$\pi_r$, where $\Si$ is a marked, punctured Riemann surface and $u$ is the holomorphic map to $\IR\times V$. 
Such fibre is isomorphic to $\bar{S}$, where $\bar{S}$ is the compact 
Riemann surface with boundary obtained from $\Si$ by compactifying each puncture to a circle, which itself corresponds to a 
copy of the moduli space $\CM_{0,1}(\gamma,\gamma)/\IR$ of cylinders over the corresponding Reeb orbit via 
the boundary gluing map. \\

Now observe that the restriction of $\LL_{r,r}$ to the fibre $F_{(u,\Si)}$ coincides with the cotangent bundle to $F_{(u,\Si)}$ 
away from the marked points. A section can be then pulled back to $\LL_{r,r}$ from the cotangent bundle itself using coordinate identification away from the marked points, where the section developes a pole of degree one. This means that a smooth section can be chosen to have a zero of index $-1$ at the marked points. With our assumption on the section in $\LL_{1,1}$ over 
each moduli space $\CM_{0,1}(\gamma,\gamma)/\IR$ we then guarantee that a coherent smooth section of $\LL_{r,r}$, when 
restricted to $F_{(u,\Si)}$, also has a singular point of index $-1$ at the punctures. In order to see this, observe that 
the gluing map at the punctures indeed agrees with the identification of $\LL_{1,1}$ with the cotangent 
bundle to $\IR\times\gamma$.\\

Moreover we will need the analogue of the following comparison formula for $\psi$-classes in Gromov-Witten theory,
$$\psi_{i,r}=\pi_{r}^*\psi_{i,r-1}+\PD[D_{i,r}],$$
where $\pi_{r}:\CM_{g,r,A}(M)\to\CM_{g,r-1}(M)$ is the map which forgets the $r$.th marked point, $\psi_{i,r}$ is the 
$i$.th $\psi$-class on $\CM_{g,r,A}(M)$ and $D_{i,r}$ is the divisor in $\CM_{g,r,A}(M)$ of nodal curves with a constant 
bubble containing only the $i$.th and $r$.th marked points. \\

In the very same way as in the proof of the comparison formula in Gromov-Witten theory, it follows that in SFT we can 
indeed choose a collection of sections $(s_{i,r})$ in such a way that, for their zero set, we have
\begin{equation}\label{comparison}
s_{i,r}^{-1}(0) = \pi_{r}^{-1}(s_{i,r-1}^{-1}(0))+D_{i,r},
\end{equation}
where here the sum in the right hand side means union with the submanifold $D_{i,r}$, transversally intersecting 
$\pi_{r}^{-1}(s_{i,r-1}^{-1}(0))$. \\

The existence of such a choice of non-generic sections follows, as in 
Gromov-Witten theory, from the fact that the pullback bundle $\pi_{r}^*\LL_{i,r-1}$ agrees with the tautological bundle 
$\LL_{i,r}$ away from the submanifold $D_{i,r}$ in $\CM_r=\CM_{g,r,A}(\Gamma^+,\Gamma^-)/\IR$, together with the fact that 
the restriction of $\LL_{i,r}$ to $D_{i,r}$ is trivial and that the normal bundle to $D_{i,r}$ agrees with $\LL_{i,r-1}$.
Notice that such a choice of sections is intrinsically non-generic, the sets $s_{i,r}^{-1}(0)$ not being smooth, but 
union of smooth components intersecting transversally. \\

We now prove that such sections can be chosen to form a coherent collection. Assume we already proved that the above choice forms a coherent collection for all the moduli spaces with up to $r-1$ marked points. Starting with a section of such collection on $\CM_{r-1}$, we construct a section on $\CM_r$ with the above configuration of zeros by first pulling back $s_{i,r-1}$ to $\CM_r \setminus D_{i,r}$ and then use the bundle map between $\pi_{r}^*\LL_{i,r-1}$ and $\LL_{i,r}$ induced by a local coordinate on the underlying curve. Such map, as we already noticed, is a bundle isomorphism on $\CM_r \setminus D_{i,r}$ and becomes singular on $D_{i,r}$: the image of $\pi_r^*s_{i,r-1}$ under this map extends to the whole $\CM_r$ assuming the value zero on $D_{i,r}$. The zero appearing this way along $D_{i,r}$ has degree $1$ by the above considerations. Moreover the section is automatically coherent, not only at the boundary components which are preimage under $\pi_r$ of boundary components of $\CM_{r-1}$, but also at the extra boundary components appearing in the fibre direction, which are always disjoint from $D_{i,r}$. Notice also that such construction works because any codimension $1$ boundary of the moduli space $\CM_{g,r,A}(\Gamma^+,\Gamma^-)$ decomposes into a product of moduli spaces where the factor containing the $i$-th marked point carries the same well defined projection map $\pi_{r}$. This is because codimension $1$ boundary strata are always formed by non-constant maps, which remain stable after forgetting a marked point.\\

The base of such induction process is given by any coherent collection of sections for the moduli spaces $\CM_1$ with only one marked point, the $i$-th, carrying the psi-class.\\

In fact coherence also requires that our choice of coherent collection of sections is symmetric with respect to permutations of the marked points (other than the $i$-th, carrying the descendant). Indeed, reiterating the construction until we forget all of the marked points but the $i$-th, we get easily
$$s_{i,r}^{-1}(0)=(\pi_{1}^*\circ\ldots\circ\hat{\pi}_i^*\circ\ldots\circ\pi_{r}^* \, s_{i,1})^{-1}(0) + \sum_{\substack{I\sqcup J=\{1,\ldots,r\}\\ \{i\}\subsetneq I \subseteq \{1,\ldots,r\}}} D^\mathrm{const}_{(I|J)}$$
where $D^\mathrm{const}_{(I|J)}$ is the submanifold of nodal curves with a constant sphere bubble carrying the marked points labeled by indices in $I$. Such choice of coherent collections of sections is indeed symmetric with respect to permutation of the marked points.\\

In order to be able to speak about higher powers $\psi_{i,r}^j$ of the psi-classes, we will need to select $j$ coherent collections of sections $s^{(k)}_{i,r}$, $k=1,\ldots,j$, intersecting transversally. The descendant moduli space $\CM^{(0,...,j,...,0)}_{g,r,A}(\Gamma^+,\Gamma^-)/\IR
\subset\CM_{g,r,A}(\Gamma^+,\Gamma^-)/\IR$ will then be given by the intersection of their zero loci. We can start from the above special non-generic choice and perturb it (preserving coherence) to $s^{(1)}_{i,r}=s_{i,r}$ and $s^{(k)}_{i,r}$, $k=2,\ldots,j$ such that the following formula holds.
$$\bigcap_{k=1}^j (s^{(k)}_{i,r})^{-1}(0) = \bigcap_{k=1}^j (\pi_{r}^* s^{(k)}_{i,r-1})^{-1}(0) \,+ \, \bigcap_{k=2}^{j}(\pi_{r}^* s^{(k)}_{i,r-1}) \,\cap\, D_{i,r}$$
Notice that, as in Gromov-Witten theory, deducing this from the comparison formula (\ref{comparison}) is possible again because the restriction of $\LL_{i,r}$ to $D_{i,r}$ is trivial. \\

Before we can use these special choices of coherent collections of sections to prove the SFT analogues of the string, dilaton and divisor equations, we however finally have to make a short comment on the genericity of our special choices. \\

Recall that for the definition of gravitational descendants in \cite{F} we need to choose sections in the tautological bundles over all moduli spaces which are generic in the sense that they are transversal to zero section, so that, in particular, all zero divisors are smooth. On the other hand, as we outlined above, all our special choices of coherent collections of sections are automatically non-generic, since their zero sets localize on nodal curves and, in particular, are not smooth. \\

In order to see that we can still use our special non-generic choices for computations, we have to make use of the fact that, by using small perturbations, the special non-generic choice of coherent collections of sections can be approximated arbitrarily close (in the $C^1$-sense) by generic coherent collections of sections. While for two different coherent collections of sections the Hamiltonian in general depends on these choices as for a given homotopy (coherent collection of sections coherently connecting the two different choices in the sense of \cite{F}) zeroes may run out of the codimension-one boundaries of the moduli spaces, we can further make use of the fact that the latter can be prevented from happening as long as the perturbation is small enough, see also the picture below. Indeed, assuming by genericity that the original section $s_0$ has no zero on the boundary and denoting by $c\neq 0$ the minimal absolute value of $s_0$ on the boundary, it is easily seen that every section $s_t$ in the homotopy has no zero on the boundary as long as $\|s_t-s_0\|_{C^1}<c/2$.\\
\begin{figure}
\begin{center}
\includegraphics[width=6cm]{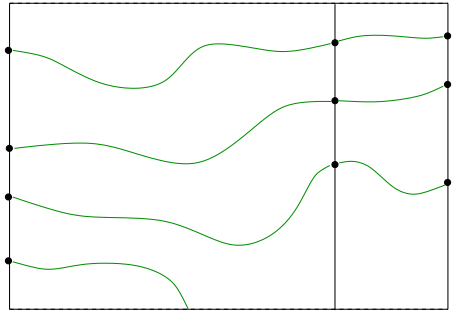}
\caption{The picture represents the trivial cobordism between a moduli space and itself (vertical black lines) and the corresponding cobordism for the zeros of coherent sections (green lines). The number of zeroes (black dots) in each copy of the moduli space may change during a homotopy (from left to right) as zeroes may run out of the codimension-one-boundary (dashed lines above and below). This, however, can be excluded as long as the homotopy is chosen sufficiently small (like the one between the middle and the right vertical lines).}
\end{center}
\end{figure}

Similarly to what happens for the gluing formulas for holomorphic curves in Floer theory (e.g. in \cite{Sch}), it then follows that the new Hamiltonians defined using these generic coherent collections of sections agree with the Hamiltonian defined using the original non-generic choices as long as the approximation error is sufficiently small, which in the gluing picture corresponds to the case of very large gluing parameter.

\vspace{0.8cm}

\subsection{Divisor equation}
As customary in Gromov-Witten theory we will assume that the chosen string of differential forms on $V$ contains a 
two-form $\theta_2$. Since by adding a marked point we increase the dimension of the moduli space by two, the integration 
of a two-form over it leaves the dimension unchanged and we can expect, as in Gromov-Witten theory, to compute the 
contributions to SFT Hamiltonian involving integration of $\theta_2$ in terms of contributions without integration, where 
the result should just depend on the homology class $A\in H_2(V)$ which can be assigned to the holomorphic curves in the 
corresponding connected component of the moduli space. \\

Recall that in order to assign an absolute homology class $A$ to a holomorphic curve $u:\Si\to\IR\times V$ we have to 
employ spanning surfaces $F_{\gamma}$ connecting a given closed Reeb orbit $\gamma$ in $V$ to a linear combination of 
circles $c_s$ representing a basis of $H_1(V)$,
\[\del F_{\gamma} = \gamma - \sum_s n_s\cdot c_s \]  
in order to define 
\[ A = [F_{\Gamma^+}] + [u(\Si)] - [F_{\Gamma^-}], \]
where $[F_{\Gamma^{\pm}}] = \sum_{n=1}^{s^{\pm}} [F_{\gamma^{\pm}_n}]$ viewed as singular chains. We might expect to find 
a result which is similar to the divisor equation in Gromov-Witten thoery whenever
\[ \int_{A}\theta_2 = \int_{u(\Si)}\theta_2, \]
that is, 
\[ \int_{F_{\Gamma^+}}\theta_2 - \int_{F_{\Gamma^-}}\theta_2 = 0 \]
which is however not satisfied, in general. \\ 

Instead of showing that it is possible to find for each class in $H^2(V)$ a nice representative which vanishes on all the 
spanning surfaces and hence meets the requirements, we want to prove a statement which holds for every chosen string of 
differential forms. Denote by $d_{\gamma}$ the integral of the differential form $\theta_2$ over the spanning surface 
of $\gamma$,
\[ d_{\gamma} = \int_{F_{\gamma}}\theta_2. \]
Denoting the $t$-variables assigned to $\theta_2$ by $t^{2,p}$ and assuming for notational simplicity that we have 
chosen a basis $A_0,...,A_N$ of $H_2(V)$ such that $\int_{A_i} \theta_2 = \delta_{0,i}$, with associated 
variables $z_0,...,z_N$, we prove the following

\begin{theorem}\label{divisor} With the above choice of non-generic coherent collections of sections, the following \emph{divisor equation} holds for the SFT Hamiltonian
$$\left(\frac{\del}{\del t^{2,0}}-z_0\frac{\del}{\del z_0}\right)\IH \;=\; \int_V t\wedge t\wedge \theta_2 + \sum_{k} t^{\alpha,k+1} c_{2\alpha}^\beta\frac{\del\IH}{\del t^{\beta, k}} +\, [\IH,\Delta],$$ 
where $c_{\gamma\alpha}^\beta$ are the structure constants of the cup product in $H^*(V)$ and where $\Delta\in\WW$ accounts for the chosen spanning surfaces and is given by
\[\Delta=\sum_{\gamma} d_{\gamma} p_{\gamma} q_{\gamma}.\]
\end{theorem}

\begin{proof} 
Using the comparison formula (\ref{comparison}), we compute, when the curve is not constant, not an orbit cylinder or 
whenever $r+|\Gamma^+|+|\Gamma^-|\geq 4$, as in the Gromov-Witten case
\begin{eqnarray*}
&&\int_{\CM^{(j_1,...,j_{r-1},0)}_{r,A}} \ev_1^*\theta_{\alpha_1}\wedge ... \wedge \ev_{r-1}^*\theta_{\alpha_{r-1}}
\wedge \ev_r^*\theta_2 \\
&&=\left(\int_{A}\theta_{2} - \int_{F_{\Gamma^+}}\theta_{2} + \int_{F_{\Gamma^-}}\theta_{2}\right) 
\int_{\CM^{(j_1,...,j_{r-1})}_{r-1,A}} \ev_1^*\theta_{\alpha_1}\wedge ... \wedge\ev_{r-1}^*\theta_{\alpha_{r-1}} \\
&&+\sum_{k=1}^{r-1} \int_{\CM^{(j_1,...,j_{r-1})}_{r-1,A}} \ev_1^*\theta_{\alpha_1}\wedge ... \wedge
\ev_k^*(\theta_2\wedge\theta_{\alpha_k})\wedge ... \wedge \ev_{r-1}^*\theta_{\alpha_{r-1}},
\end{eqnarray*}
where $\CM^{(j_1,...,j_r)}_{r,A}= \CM^{(j_1,...,j_r)}_{g,r,A}(\Gamma^+,\Gamma^-)/\IR$ denotes the component of the moduli 
space of curves representing the homology class $A\in H_2(V)$. Note that since we can assume that the Hamiltonian counts 
holomorphic curves with at least one puncture, we do not get contributions from constant curves. On the other hand, when the curve is constant and 
$r+|\Gamma^+|+|\Gamma^-|=3$ the integral is given by $\int_V t\wedge t\wedge \theta_2$ and in the case of orbit cylinders 
with only one marked point any correlator involving only a $2$-form vanishes for dimensional reasons. \\

Notice now that the differential operator multiplying each monomial containing $p^{\Gamma^+}q^{\Gamma^-}$ in $\IH$ by the coefficient
$$\int_{F_{\Gamma^+}}\theta_{2} - \int_{F_{\Gamma^-}}\theta_{2}$$
is precisely
$$\sum_{\gamma}\left(d_{\gamma}p_{\gamma}\frac{\del}{\del p_{\gamma}} -d_{\gamma}q_{\gamma}\frac{\del}{\del q_{\gamma}}\right)$$
This, together with 
\[ \sum_{\gamma}\left(d_{\gamma}p_{\gamma}\frac{\del\IH}{\del p_{\gamma}} -d_{\gamma}q_{\gamma}\frac{\del\IH}{\del q_{\gamma}}\right)
   = [\IH,\Delta] \]
yields the desired equation. 
\end{proof}

Note that even when we restrict to special choices for the differential forms and coherent collections of sections, the Hamiltonian 
$\IH$ itself still depends on all other choices like contact form, cylindrical almost complex structure and so on. With the above main 
application in mind it is even more important that we have the following

\begin{corollary}\label{cdivisor} For any choice of differential forms and coherent collections of sections the following \emph{divisor equation} holds when passing to SFT-homology 
\[\left(\frac{\del}{\del t^{2,0}} -z_0\frac{\del}{\del z_0}\right)\IH\;=\; \int_V t\wedge t\wedge \theta_2 + \sum_{k} t^{\alpha,k+1} c_{2\alpha}^\beta\frac{\del\IH}{\del t^{\beta, k}} \hspace{0.5cm} \;\in\; H_*(\hbar^{-1}\WW,[\IH,\cdot]),\] 
\end{corollary}

\begin{proof}
First it follows from $[\IH,\Delta]=0\in H_*(\hbar^{-1}\WW,[\IH,\cdot])$ that the equation on SFT homology holds for our 
special choice of coherent collections of sections, in particular, is independent of the auxiliary choice of spanning surfaces in order 
to assign absolute homology classes to punctured holomorphic curves. \\

We redenote by $\IH^+$ the Hamiltonian used in theorem \ref{divisor} and coming 
from the special choice of coherent collections of sections and auxiliary data we made there. To prove that the desired equation holds up to homology for any choice of 
coherent collections of sections and any other auxiliary data, leading to a new Hamiltonian $\IH^-$, we just need to check that its 
terms are properly covariant with respect to the isomorphism $F^-_*\circ (F^+_*)^{-1}: H_*(\hbar^{-1}\WW^+,[\IH^+,\cdot])
\to H_*(\hbar^{-1}\WW^-,[\IH^-,\cdot])$. \\

Indeed it more generally follows from the computation at the end of the previous 
section that, if $D$ is any first order graded differential operator in the $t$ and $z$ variables, then we have
$(F^-_*\circ(F^+_*)^{-1})(D\IH^+) = D\IH^-$, so that in particular
$\IH^+_{\alpha,p}=D\IH^+\in H_*(\hbar^{-1}\WW^+,[\IH^+,\cdot])$  
implies $\IH^-_{\alpha,p}=D\IH^-\in H_*(\hbar^{-1}\WW^-,[\IH^-,\cdot])$. \\

To be more precise, this follows from the fact that $D$ given by 
\[ D = z_0\frac{\del}{\del z_0} + \sum_{k} t^{\alpha,k+1} c_{2\alpha}^\beta\frac{\del}{\del t^{\beta, k}} \]
satisfies like $\del/\del t_{\alpha,p}$ the (graded) Leibniz rule, that is, 
we have the two identities 
\begin{equation*} [\IH,D\IH]= 0, \end{equation*} 
so that $D\IH\in H_*(\hbar^{-1}\WW,[\IH,\cdot])$, and, if $\IF$ is the potential for the cobordism connecting the different choices of auxiliary data, 
\begin{eqnarray*}
&& e^{\IF}(D\overleftarrow{\IH}^+)e^{-\IF} - e^{-\IF}(D\overrightarrow{\IH}^-)e^{\IF}\\
&& + (e^{\IF}D\IF) \overleftarrow{\IH}^+ e^{-\IF} + e^{\IF} \overleftarrow{\IH}^+ e^{-\IF}D\IF 
- e^{-\IF}D\IF \overrightarrow{\IH}^- e^{\IF} + e^{-\IF} \overrightarrow{\IH}^- (e^{\IF}D\IF) \\
&& = D(e^{\IF} \overleftarrow{\IH}^+ e^{-\IF} - e^{-\IF} \overrightarrow{\IH}^- e^{\IF}) = 0, 
\end{eqnarray*}
which implies as before $F^+_*(D\IH^+) = F^-_*(D\IH^-)$. For the computations note that the degree of $D$ is zero and hence even. Finally the term accounting for constant curves is even invariant as it is mapped to itself by $F^-_*\circ(F^+_*)^{-1}$.\end{proof}

Note that when we specialize to $t=0$ the above equation simplifies to 
\begin{equation*} \IH^1_{2,0} = \frac{\del \IH^0}{\del z_0} \;\in\; H_*(\hbar^{-1}\WW^0,[\IH^0,\cdot])
\end{equation*} 
and hence allows for the computation of one of the Hamiltonians $\IH^1_{\alpha,p}\in 
H_*(\hbar^{-1}\WW^0,[\IH^0,\cdot])$ in terms of the Hamiltonian $\IH^0$ counting holomorphic curves without marked points. \\

{\bf Remark:} If the dimension of $V$ is large enough, we indeed find for every $\theta\in\Omega^2(V)$ another 
differential 2-form $\theta^0$ with $[\theta^0]=[\theta]\in H^2(V)$ which vanishes on all the spanning surface 
$F_{\gamma}$. Under the assumption that all the spanning surfaces can be chosen to be embedded and pairwise disjoint, 
which leads to the requirement on the dimension of $V$, the statement follows by modifying the differential form 
inductively after proving it for the spanning surface of a single orbit $\gamma$.  Indeed, for chosen 
$\theta\in\Omega^2(V)$ let $\theta_{\gamma} = \iota_{\gamma}^*\theta$ denote the pullback under the embedding of 
$F_{\gamma}$ into $V$. Since every 2-form on a surface with boundary is neccessarily exact, we can choose a (primitive) 
1-form $\lambda_{\gamma}\in\Omega^1(F_{\gamma})$ with $\theta_{\gamma}=d\lambda_{\gamma}$ which we extend to a one-form 
$\lambda$ on $V$ with support only in a small neighborhood of $F_{\gamma}$. Since 
$\iota_{\gamma}^*(\theta - d\lambda) = \theta_{\gamma} - d\lambda_{\gamma} = 0$, it follows that 
$\theta^0:=\theta-d\lambda$ meets the desired requirements.

\vspace{0.5cm}

\subsection{Dilaton equation}
The next equation we will study is the dilaton equation.
\begin{theorem}
For any choice of coherent collections of sections the following \emph{dilaton equation} holds for the SFT Hamiltonian when passing to 
SFT-homology$$\frac{\del}{\del t^{0,1}}\IH\;=\; \ID_{\mathrm{Euler}}\IH  \;\in\, H_*(\hbar^{-1}\WW,[\IH,\cdot])$$
with the first-order differential operator 
$$\ID_{\mathrm{Euler}} := -2\hbar\frac{\del}{\del\hbar}-\sum_\gamma p_\gamma\frac{\del}{\del p_\gamma}
-\sum_\gamma q_\gamma\frac{\del}{\del q_\gamma}-\sum_{\alpha,p}t^{\alpha,p}\frac{\del}{\del t^{\alpha,p}}.$$
The same equation holds at the chain level for the above special choice of non-generic coherent collections of sections.
\end{theorem}

\begin{proof}
With our special choice of non-generic coherent collections of sections still standing, the proof is precisely the same as in 
Gromov-Witten theory. We want to compute the integral
$$\int_{\CM^{(j_1,...,j_{r-1},1)}_r} \ev_1^*\theta_{i_1}\wedge ... \wedge\ev_{r-1}^*\theta_{i_{r-1}}.$$
Notice that the tautological bundle $\LL_{r,r}$ restricted on the fibre of the forgetful fibration $\pi_r$ coincides 
with $\omega+z_1+\ldots+z_r$, where $\omega$ is the canonical bundle and $z_1,\ldots,z_r$ are the marked points. Since 
the generic fiber is a smooth curve with $|\Gamma^+|+|\Gamma^-|$ holes and since, by our proper choice of sections for 
$\LL_{1,1}$ on the simplest moduli space of orbit cylinders with one marked point, coherence at such holes is equivalent 
to closing the holes and imposing an extra pole there, we can argue in the very same way as in Gromov-Witten theory. \\

Finally we would need to separately consider the cases where the forgetful fibration $\pi_r$ is not defined: as in 
Gromov-Witten theory only constant curves of genus one with one marked point might give a contribution, but in SFT such moduli space 
has virtual dimension one and we hence get no contribution by index reasons. Translating this into differential operators on the 
Hamiltonian yields the desired equation. \\

To prove that the same equation holds for any choice of auxiliary data when passing to SFT-homology we need to check 
covariance of the right hand side with respect to $F^-_*\circ (F^+_*)^{-1}: H_*(\hbar^{-1}\WW^+,[\IH^+,\cdot])\to 
H_*(\hbar^{-1}\WW^-,[\IH^-,\cdot])$, as in corollary \ref{cdivisor}. This time $\ID_{\mathrm{Euler}}$ is not a first order 
differential operator in the $t$ and $z$ variables, but also involves $p$ and $q$ variables and the variable $\hbar$ for the 
genus. \\

While all but the last summands of $\ID_{\mathrm{Euler}}$,
$$\ID_{\mathrm{Euler}} =  - 2\hbar\frac{\del}{\del\hbar} - \sum_\gamma p_\gamma\frac{\del}{\del p_\gamma}
- \sum_\gamma q_\gamma\frac{\del}{\del q_\gamma}-\sum_{\alpha,p}t^{\alpha,p}\frac{\del}{\del t^{\alpha,p}}$$ 
do not satisfy the desired Leibniz rule with respect to the bracket, the sum operator $\ID_{\mathrm{Euler}}$ 
has the desired property thanks to the fact that it extracts the Euler characteristic of the corresponding curves 
from each monomial in the variables 
$t,p,q,\hbar$. \\

Indeed, additivity of the Euler characteristic with respect to gluing straightforwardly shows that 
$\ID_{\mathrm{Euler}}$ satisfies the Leibniz rule, that is, as in the proof of the divisor equation 
we have the two identities 
\begin{equation*} [\IH,\ID_{\mathrm{Euler}}\IH]=\ID_{\mathrm{Euler}}[\IH,\IH] = 0, \end{equation*} 
so that $\ID_{\mathrm{Euler}}\IH\in H_*(\hbar^{-1}\WW,[\IH,\cdot])$, and, if $\IF$ is the potential for the cobordism connecting 
the different choices of auxiliary data, 
\begin{eqnarray*}
&& e^{\IF}(\ID_{\mathrm{Euler}}\overleftarrow{\IH}^+)e^{-\IF} - e^{-\IF}(\ID_{\mathrm{Euler}}\overrightarrow{\IH}^-)e^{\IF}\\
&& + (e^{\IF}\ID_{\mathrm{Euler}}\IF) \overleftarrow{\IH}^+ e^{-\IF} + e^{\IF} \overleftarrow{\IH}^+ e^{-\IF}\ID_{\mathrm{Euler}}\IF \\
&& - e^{-\IF}\ID_{\mathrm{Euler}}\IF \overrightarrow{\IH}^- e^{\IF} + e^{-\IF} \overrightarrow{\IH}^- (e^{\IF}\ID_{\mathrm{Euler}}\IF) \\
&& = \ID_{\mathrm{Euler}}(e^{\IF} \overleftarrow{\IH}^+ e^{-\IF} - e^{-\IF} \overrightarrow{\IH}^- e^{\IF}) = 0, 
\end{eqnarray*}
which implies as before $F^+_*(\ID_{\mathrm{Euler}} \IH^+) = F^-_*(\ID_{\mathrm{Euler}} \IH^-)$. 
\end{proof}

Note that when we specialize to $t=0$ the above equation yields the identity 
\begin{equation*} \IH^1_{0,1} = \ID_{\mathrm{Euler}}\IH^0 \;\in\; H_*(\hbar^{-1}\WW^0,[\IH^0,\cdot])
\end{equation*} 
and hence allows for the computation of a second one of the Hamiltonians $\IH^1_{\alpha,p}\in 
H_*(\hbar^{-1}\WW^0,[\IH^0,\cdot])$ in terms of the original Hamiltonian $\IH^0$ counting holomorphic curves without marked points. \\
\vspace{0.5cm}

\subsection{String equation}
It just remains to understand how the string equation translates from Gromov-Witten theory to SFT. Indeed string equation 
is an even more straightforward application of the comparison formula (\ref{comparison}) and, reasoning along the same 
line as in the proof of divisor equation (included the covariance statement), we easily get the following theorem.

\begin{theorem}
For any choice of coherent collections of sections the following \emph{string equation} holds for the SFT Hamiltonian when passing to SFT-homology
$$\frac{\del}{\del t^{0,0}}\IH\;=\; \int_V t\wedge t  + \sum_{k}t^{\alpha,k+1}\frac{\del}{\del t^{\alpha,k}}\IH 
 \;\in\; H_*(\hbar^{-1}\WW,[\IH,\cdot]),$$ 
The same equation holds at the chain level for the above special choice of non-generic coherent collections of sections.
\end{theorem}

Observe that when we specialize to $t=0$ we now get the obvious result $\IH^1_{0,0}= 0$.

\end{document}